\documentclass[reqno]{amsart}

\usepackage{amsmath}
\usepackage{amsfonts}
\usepackage{amsthm}
\usepackage{amssymb}
\usepackage{graphicx}
\usepackage{graphics}
\usepackage{bm}
\usepackage{dsfont}
\usepackage[all]{xy}
\usepackage{color}
\usepackage[font=footnotesize]{caption}
\numberwithin{equation}{section}
\newtheoremstyle{personal}%
{12pt}
{12pt}
{\slshape}
{}
{\bfseries}
{.}
{.5em}
{}
\theoremstyle{personal}%
\newtheorem{thm}{Theorem}[section]

\newtheorem{cor}[thm]{Corollary}
\newtheorem{lem}[thm]{Lemma}

\theoremstyle{definition}
\newtheorem{rem}[thm]{Remark}

\definecolor{gray}{gray}{0.4}

\newcommand{\N}{\mathds{N}}
\newcommand{\Z}{\mathds{Z}}
\newcommand{\R}{\mathds{R}}
\newcommand{\Q}{\mathds{Q}}
\newcommand{\C}{\mathds{C}}

\newcommand{\K}{\mathds{K}}

\newcommand{\diff}{\mathrm{d}}

\newcommand{\Tan}{\mathrm{T}}
\newcommand{\cu}{c_{\mathrm{u}}}
\newcommand{\cw}{c_{\mathrm{w}}}

\newcommand{\Aubry}{\mathcal{A}}
\newcommand{\AC}{\mathrm{AC}}
\newcommand{\SSS}{\mathcal{S}}

\newcommand{\W}{W^{1,2}}
\newcommand{\Wc}{W^{\mathrm{c}}}

\newcommand{\CC}{\mathcal{C}}

\newcommand{\UU}{\mathcal{U}}

\newcommand{\WW}{\mathcal{W}}
\newcommand{\MM}{\mathcal{M}}
\newcommand{\PP}{\mathcal{P}}

\newcommand{\crit}{\mathrm{crit}}
\newcommand{\ind}{\mathrm{ind}}
\newcommand{\nul}{\mathrm{nul}}

\begin{document}

\title{Waist theorems for Tonelli systems in higher dimensions}

\author[L. Asselle]{Luca Asselle}
\address{Luca Asselle\newline\indent
Justus Liebig Universit\"at Gie\ss en, Mathematisches Institut \newline\indent Arndtstra\ss e 2, 35392 Gie\ss en, Germany}
\email{luca.asselle@ruhr-uni-bochum.de}

\author[M. Mazzucchelli]{Marco Mazzucchelli}
\address{Marco Mazzucchelli\newline\indent 
CNRS, \'Ecole Normale Sup\'erieure de Lyon, UMPA\newline\indent  
46 all\'ee d'Italie, 69364 Lyon Cedex 07, France}
\email{marco.mazzucchelli@ens-lyon.fr}

\date{July 3, 2018. \emph{Revised}: September 17, 2019.}
\subjclass[2000]{37J45, 58E05}
\keywords{Tonelli Lagrangian, Ma\~n\'e critical values, Aubry set, periodic orbit, waist}

\begin{abstract}
We study the periodic orbits problem on energy levels of Tonelli Lagrangian systems over configuration spaces of arbitrary dimension. We show that, when the fundamental group is finite and the Lagrangian has no stationary orbit at the Ma\~n\'e critical energy level, there is a waist on every energy level just above the Ma\~n\'e critical value. With a suitable perturbation with a potential, we show that there are infinitely many periodic orbits on every energy level just above the Ma\~n\'e critical value, and on almost every energy level just below. Finally, we prove the Tonelli analogue of a closed geodesics result due to Ballmann-Thorbergsson-Ziller.
\end{abstract}

\maketitle

\section{Introduction}

Over the last few years there have been significant advances in the study of the multiplicity of periodic orbits with low energy of Tonelli Lagrangian systems over 2-dimensional closed configuration spaces \cite{Taimanov:1992sm, Contreras:2004lv, Abbondandolo:2015lt, Abbondandolo:2014rb, Asselle:2015ij, Asselle:2015sp, Abbondandolo:2016bh, Asselle:2018}. So far, there are essentially no analogous results for higher dimensional configuration spaces; at best, we know the existence of at least one periodic orbit for almost all low energy levels, thanks to the work of Contreras \cite{Contreras:2006yo}.

We recall that a Tonelli Lagrangian is a smooth function $L:\Tan M\to\R$ defined over the tangent bundle of a closed manifold $M$ (the configuration space) that is fiberwise superlinear with positive-definite fiberwise Hessian, see, e.g., \cite{Contreras:1999fm, Fathi:2008xl}. The phase space $\Tan M$ is laminated into the level sets of the energy function $E:\Tan M\to\R$, $E(q,v)=\partial_vL(q,v)v-L(q,v)$. The Lagrangian defines a flow $\phi_L^t:\Tan M\to\Tan M$ that preserves each compact energy level $E^{-1}(e)$. The flow lines have the form $\phi^t(\gamma(0),\dot\gamma(0))=(\gamma(t),\dot\gamma(t))$, where $\gamma:\R\to M$ is a solution of the Euler-Lagrange equation of $L$. A periodic curve $\gamma:\R/p\Z\to M$ lifts to a periodic orbit of the Euler-Lagrange flow on the energy level $E^{-1}(e)$ if and only if it is a critical point of the free-period action functional 
\begin{align*}
\SSS_e(\gamma)=\int_0^p L(\gamma(t),\dot\gamma(t))\,\diff t + pe.
\end{align*}
Here, $\SSS_e$ is defined on the space of periodic curves of any possible positive period. This space is formally given by $\MM:=\W(\R/\Z,M)\times(0,\infty)$, so that a pair $(\Gamma,p)\in\MM$ defines the $p$-periodic curve $\gamma(t)=\Gamma(t/p)$. Among the critical points of $\SSS_e$, the local minimizers, which are figuratively called \textbf{waists}, force important consequences on the Euler-Lagrange dynamics, at least when $M$ is a surface.

The variational properties of the free-period action functional $\SSS_e$ are similar to the ones of the geodesic energy from Finsler geometry when $e$ is large, whereas several difficulties arise when $e$ is low. What mark the boundary between the high and low energies are the so-called Ma\~n\'e critical value
\begin{align*}
c(L):=\min\big\{e\in\R\ \big|\ \SSS_e(\gamma)\geq 0\ \forall\gamma\in\MM\big\},
\end{align*}
and the analogous critical values $\cu(L)$ and $c_0(L)$ of the lifts of the Tonelli Lagrangian $L$ to the tangent bundle of the universal cover and of the universal abelian cover of $M$ respectively. Another relevant energy value is $e_0(L):=\max E(\cdot,0)$, above which every energy level covers the whole configuration space $M$. These energy values are ordered as $e_0(L)\leq \cu(L)\leq c_0(L) \leq c(L)$. The inequality  $e_0(L)\leq\cu(L)$ is strict for a suitably generic Tonelli Lagrangian $L$ (see \cite[Prop.~4.2]{Asselle:2018}), whereas the inequalities $\cu(L)\leq c_0(L) \leq c(L)$ become equalities in particular when $M$ has finite fundamental group. We refer the reader to \cite{Contreras:1999fm, Abbondandolo:2013is} for the background on the Ma\~n\'e critical values.

On every energy level $e>c_0(L)$, the Euler-Lagrange dynamics is conjugated to a Finsler geodesic flow on the unit tangent bundle of $M$. The problem of periodic orbits in this energy range is reduced to the famous closed geodesic problem for closed Finsler manifolds. Even more so, a waist with energy $e>c_0(L)$ of the Lagrangian system correspond to a waist of the associated Finsler metric, that is, a closed geodesic that locally minimizes the Finsler length among nearby periodic curves. If the manifold $M$ is simply connected, a Finsler metric on $M$ does not necessarily have waists. If $M$ has non-trivial (and possibly even finite) fundamental group, a Finsler metric on $M$ always has waists in non-trivial homotopy classes, but it does not necessarily have contractible ones; 
nevertheless, when $M$ is a 2-sphere, a recent results of the authors together with Benedetti \cite[Theorem~1.3]{Asselle:2018} guarantees that a contractible simple waist always exists on every energy level just above the Ma\~n\'e critical value $c(L)$. Our first theorem extends the validity of \cite[Theorem~1.3]{Asselle:2018}, except for the simplicity of the waists, to higher dimensional closed configuration spaces. This will later allow us to apply some techniques from 2-dimensional Tonelli dynamics to study the multiplicity of periodic orbits in generic Tonelli Lagrangian systems on arbitrary configuration spaces.

\begin{thm}\label{t:contractible_waist}
Let $M$ be a closed manifold with finite fundamental group, and $L:TM\rightarrow \R$ a Tonelli Lagrangian such that $e_0(L)<c(L)$. There exists $c_{\text w}(L)>c(L)$ such that, for every $e\in (c(L),\cw(L))$, $L$ possesses at least a contractible waist with energy $e$. 
\end{thm}

On orientable surfaces, the existence of a contractible waist has strong consequences: it often forces the existence of infinitely many other contractible periodic orbits on the same energy level, or on arbitrarily close energy levels. In higher dimension this phenomenon still exists, but requires stronger assumptions on the waist, for instance the hyperbolicity. This latter property is not generic. However, a theorem of Carballo-Miranda \cite{Carballo:2013fc}, which extends a result of Klingenberg-Takens \cite{Klingenberg:1972dp} for geodesic flows, implies that, by perturbing a Tonelli Lagrangian with a $C^\infty$ generic potential, each periodic orbit in a given energy level is either hyperbolic or of twist type. In the twist case, the celebrated Birkhoff-Lewis Theorem \cite[Theorem~3.3.A.1]{Klingenberg:1978so} guarantees that the periodic orbit is an accumulation of periodic orbits on the same energy level. Arguing along this line, we obtain the following result, which we state as a corollary of Theorem~\ref{t:contractible_waist} since the infinitely many periodic orbits found are forced to exist by the waist provided by Theorem~\ref{t:contractible_waist}.

\begin{cor}\label{c:supercritical_mult}
Let $M$ be a closed manifold with finite fundamental group, and $L:TM\rightarrow \R$ a Tonelli Lagrangian. There exists a residual subset $\UU\subset C^\infty(M)$ such that, for all $U\in\UU$ satisfying $e_0(L-U)<c(L-U)$, the Lagrangian $L-U$ possesses infinitely many  periodic orbits on every energy level $e$ in an open dense subset of $(c(L-U),\cw(L-U))$.
\end{cor}

We wish to stress that, according to the already mentioned \cite[Prop.~4.2]{Asselle:2018}, the space of Tonelli Lagrangians $L:TM\rightarrow \R$ satisfying the strict inequality $e_0(L)<c(L)$, or even the stronger one $e_0(L)<\cu(L)$, is $C^0$-open and $C^1$-dense in the whole space of Tonelli Lagrangians. In particular, if $L$ satisfies $e_0(L)<c(L)$, the same inequality holds for $L-U$ provided the potential $U$ is $C^0$-small.

An analogous statement also allows us to deal with almost all energy levels just below the Ma\~n\'e critical value and with all energy levels just above. 

\begin{thm}\label{t:just_below}
Let $M$ be a closed manifold with finite fundamental group, and $L:TM\rightarrow \R$ a Tonelli Lagrangian such that $e_0(L)<c(L)$. There exist an arbitrarily $C^1$-small $U\in C^\infty(M)$ and $\epsilon>0$ such that $c(L)=c(L-U)$, and $L-U$ possesses infinitely many periodic orbits on almost every energy level $e\in(c(L)-\epsilon,c(L))$ and on every energy level $e\in(c(L),c(L)+\epsilon)$.
\end{thm}

The proof of Theorem~\ref{t:just_below} combines the 2-dimensional techniques in \cite{Abbondandolo:2014rb} with a closing lemma in Aubry-Mather theory established by Figalli-Rifford \cite{Figalli:2015ft}. More specifically, the perturbation of the Lagrangian provides a hyperbolic periodic orbit that is the whole Aubry set, and in particular is a waist.

In his last paper \cite{Mane:1996qe} Ma\~n\'e conjectured that the Aubry set of a Tonelli Lagrangian perturbed with a $C^r$ generic potential, for some $r\geq 2$, consists of a (possibly stationary) hyperbolic periodic orbit. If the Ma\~n\'e conjecture were true, Theorem~\ref{t:just_below} could be strengthen by allowing to perturb $L$ with a $C^\infty$ generic potential $U$ such that $e_0(L-U)<c(L-U)$; in particular, the obtained statement would improve Corollary~\ref{c:supercritical_mult}. Recently, Ma\~n\'e's conjecture was proved by Contreras \cite{Contreras:2014ul} in the case where $r=2$ and $M$ is a closed surface. However, when $M$ is a closed surface, much stronger statements than Theorem~\ref{t:just_below} hold: there are infinitely many periodic orbits on almost all energy levels in $(e_0(L),\cu(L))$, and when $M$ is a sphere and $(e_0(L),c(L))\neq\varnothing$, there are infinitely many periodic orbits on every energy level just above $c(L)$, see \cite{Abbondandolo:2014rb, Asselle:2016qv, Asselle:2018}.

Our last perturbative result is the Tonelli generalization of a closed geodesics theorem due to Ballmann-Thorbergsson-Ziller \cite{Ballmann:1981qi}. Our extension encompasses the Finsler case, but also includes the possibly non-empty energy range between the Ma\~n\'e critical value of the universal cover and of the universal abelian cover, where the dynamics is not conjugate to  a Finsler geodesic flow. We denote by $[\alpha]$ the conjugacy class of an element $\alpha\in\pi_1(M)$. We recall that the conjugacy classes of the fundamental group $\pi_1(M)$ are in one-to-one correspondence with the connected components of the free loop space of $M$, see, e.g., \cite[Prop.~3.2.2]{Mazzucchelli:2012nm}.

\begin{thm}\label{t:btz}
Let $M$ be a closed manifold having a non-trivial $\alpha\in \pi_1(M)$ satisfying $[\alpha^n]=[\alpha ^m]$ for some distinct non-negative integers $n,m$. For each Tonelli Lagrangian $L:\Tan M\to\R$ there exists a residual subset $\UU\subset C^\infty(M)$ such that, for all $U\in\UU$, the Lagrangian $L-U$ possesses infinitely many periodic orbits on every energy level $e$ in an open dense subset of $(\cu(L-U),\infty)$. Moreover, for each infinite subset $\K\subseteq\N$, such periodic orbits can be found within the homotopy classes $\big\{[\alpha^{kn}]\ \big|\ k\in\K\big\}$.
\end{thm}

Notice that, if $\alpha\in\pi_1(M)$ is a non-trivial element of finite order $n-1$, and thus $\alpha^{n}=\alpha$, we can choose $\K:=\big\{n^h\ \big|\ h\in\N \big\}$.
Since $\alpha^{kn}=\alpha$ for all $k\in\K$, Theorem~\ref{t:btz} implies that $L-U$ possesses infinitely many periodic orbits in the homotopy class $[\alpha]$ on every energy level $e$ in an open dense subset of $(\cu(L-U),\infty)$.

The integers $n,m$ were not required to be non-negative in the original statement for Riemannian geodesic flows in \cite{Ballmann:1981qi}. In our proof, such an assumption is needed due to the possible non-reversibility of the Tonelli Lagrangian $L$. We discuss this further in Remark~\ref{r:reversibility}.

\subsection*{Organization of the paper}
In Section~\ref{s:waists} we prove Theorem~\ref{t:contractible_waist}. In Section~\ref{s:hyp_twist} we give the needed background on generic Hamiltonian dynamics, and in particular recall some properties of hyperbolic periodic orbits and of periodic orbits of twist type. In Section~\ref{s:mult}, we prove the remaining statements.

\subsection*{Acknowledgments}
Marco Mazzucchelli is grateful to Alessio Figalli and Ludovic Rifford for pointing out the argument in \cite[page~935]{Contreras:1999wj} in order to obtain the hyperbolicity of the periodic orbit in their result \cite[Theorem~1.2]{Figalli:2015ft}. Both authors are grateful to the anonymous referee for her/his careful reading of the paper, and for spotting a few inaccuracies in the first draft.
Luca Asselle is partially supported by the DFG-grants AB 360/2-1 ``Periodic orbits of conservative systems below the Ma\~n\'e critical energy value'' and AS 546/1-1 ``Morse theoretical methods in Hamiltonian dynamics''.

\section{Supercritical waists}
\label{s:waists}

We begin by recalling some results on Aubry-Mather theory from Ma\~n\'e's perspective. We refer the reader to \cite{Mane:1997nw, Contreras:1997jq} for the proofs. Let $M$ be a closed manifold, and $L:\Tan M\to\R$ a Tonelli Lagrangian with associated energy function $E:\Tan M\to\R$, $E(q,v)=\partial_vL(q,v)v-L(q,v)$. Given two points $q_0,q_1\in M$, we denote by $\AC(q_0,q_1)$ the space of all absolutely continuous curves $\gamma:[0,\tau]\to M$ defined on any compact interval of the form $[0,\tau]$ and such that $\gamma(0)=q_0$ and $\gamma(\tau)=q_1$. We stress that the parameter $\tau\geq0$ is not fixed, and thus different curves in $\AC(q_0,q_1)$ are defined on possibly different intervals. The action at energy $e\in\R$ of an absolutely continuous curve $\gamma:[0,\tau]\to M$ is the quantity
\begin{align*}
\SSS_e(\gamma)=\int_0^\tau L(\gamma(t),\dot\gamma(t))\,\diff t + \tau e\in\R\cup\{+\infty\}.
\end{align*}
The action potential at energy $e$ is the function
\begin{align*}
\Phi_e:M\times M\to\R\cup\{-\infty\},
\qquad
\Phi_e(q_0,q_1)=\inf \big\{\SSS_e(\gamma)\ \big|\ \gamma\in\AC(q_0,q_1) \big\}.
\end{align*}
The Ma\~n\'e critical energy value $c(L)$ is precisely the minimum $e\in\R$ such that $\Phi_e$ is uniformly bounded from below. In particular 
\[
\Phi_{c(L)}(q_0,q_1)+\Phi_{c(L)}(q_1,q_0)\geq \Phi_{c(L)}(q_0,q_0)=0,
\qquad
\forall q_0,q_1\in M,
\] 
and therefore $\SSS_{c(L)}(\gamma)\geq-\Phi_{c(L)}(q_1,q_0)$ for all $\gamma\in\AC(q_0,q_1)$. The Aubry set $\Aubry(L)$ is defined as the set of those $(q,v)\in\Tan M$ with the following property: if $\gamma:\R\to M$ is a solution of the Euler-Lagrange equation of $L$ with $\gamma(0)=q$ and $\dot\gamma(0)=v$, then for each $\tau\geq0$ we have $\SSS_{c(L)}(\gamma|_{[0,\tau]})=-\Phi_{c(L)}(\gamma(\tau),\gamma(0))$. It turns out that the Aubry set $\Aubry(L)$ is always non-empty and contained in the energy level $E^{-1}(c(L))$.

Now, consider the free-period action functional at energy $e$, which is simply the restriction of the action $\SSS_e$ over the space of $\W$ periodic curves of any positive period. Formally, such a space is given by the product \[\MM:=\W(\R/\Z,M)\times(0,\infty),\] and a pair $(\Gamma,p)\in\MM$ will be identified with the $p$-periodic curve $\gamma:\R/p\Z\to M$, $\gamma(t)=\Gamma(t/p)$. As customary in the literature, we will write $\gamma=(\Gamma,p)$. The free-period action functional is then
$\SSS_e:\MM\to\R\cup\{\infty\},\  \SSS_e(\Gamma,p)=\SSS_e(\gamma)$.
Its critical points are precisely those $\gamma=(\Gamma,p)$ that are periodic solutions of the Euler-Lagrange equation of $L$ with energy $E(\gamma,\dot\gamma)\equiv e$. Throughout the paper, we can always assume without loss of generality that our Tonelli Lagrangian $L$ is fiberwise quadratic outside a large compact subset of $\Tan M$ containing the energy levels that we consider (see \cite[Appendix~A]{Asselle:2018}). In this way, the free-period action functional  is everywhere finite, $C^{1,1}$, and twice Gateaux differentiable. For each $[p_1,p_2]\subset(0,\infty)$, its restriction $\SSS_e|_{\W(\R/\Z,M)\times[p_1,p_2]}$ satisfies the Palais-Smale condition with a suitable choice of complete Riemannian metric on $\MM$. If $e>\cu(L)$, also the unrestricted functional $\SSS_e$ satisfies the Palais-Smale condition. We refer the reader to \cite{Abbondandolo:2013is} for the proof of these properties. 

Theorem~\ref{t:contractible_waist} will be a consequence of the following lemma.

\begin{lem}
\label{l:waist}
If the Aubry set $\Aubry(L)$ does not contain $\tau$-periodic orbits for some $\tau>0$, then there exists $\cw(L,\tau)>c(L)$ such that, for every $e\in (c(L),\cw(L,\tau))$, $L$ possesses at least a  waist with energy $e$ and period larger than~$\tau$.
\end{lem}

\begin{rem}
The waist with period larger than $\tau$ provided by Lemma~\ref{l:waist} may well be the iterate of another periodic orbit (which must also be a waist) whose period is less than $\tau$. Nevertheless, the lower bound $\tau$ for the period of the waist is a significant information when $M$ is not a closed orientable surface, because in such a case it is possible that certain iterates of a waist are not themselves waists.
\end{rem}

\begin{proof}[Proof of Lemma~\ref{l:waist}]
If $\tau>0$ is such that there are no $\tau$-periodic orbits in the Aubry set $\Aubry(L)$, we claim that 
\begin{align*}
\delta:=\inf \SSS_{c(L)}|_{\W(\R/\Z,M)\times\{\tau\}}>0. 
\end{align*}
Indeed, since the restriction of $\SSS_{c(L)}$ to $\W(\R/\Z,M)\times\{\tau\}$ satisfies the Palais-Smale condition, there exists $\gamma=(\Gamma,\tau)\in\W(\R/\Z,M)\times\{\tau\}$ such that 
$\SSS_{c(L)}(\gamma)=\delta$.
If $\delta=0$, since $\SSS_{c(L)}$ is a non-negative function, we would have $\SSS_{c(L)}(\gamma)=\inf \SSS_{c(L)}=0$. Therefore, $\gamma$ would be a $\tau$-periodic orbit of $L$ with energy $c(L)$, which is impossible by our choice of $\tau$.

Now, fix a point $(q,v)\in\Aubry(L)$, and consider the solution $\gamma:[0,\tau]\to M$ of the Euler-Lagrange equation of $L$ such that $\gamma(0)=q$ and $\dot\gamma(0)=v$. Since $\SSS_{c(L)}(\gamma)+\Phi_{c(L)}(\gamma(\tau),\gamma(0))=0$, we can extend $\gamma$ to a $\W$ periodic curve $\gamma:\R/p\Z\to M$, for some $p>\tau$, such that $\SSS_{c(L)}(\gamma)<\delta/2$. For each energy value $e\in(c(L),c(L)+\tfrac{\delta}{2p})$, we have 
\begin{align*}
\SSS_e(\gamma)=\SSS_{c(L)}(\gamma)+(e-c(L))p<\delta<\delta+(e-c(L))\tau=\inf \SSS_{e}|_{\W(\R/\Z,M)\times\{\tau\}}, 
\end{align*}
and therefore
\begin{align*}
\delta_0:=\inf \SSS_{e}|_{\W(\R/\Z,M)\times(\tau,\infty)} < \inf \SSS_{e}|_{\W(\R/\Z,M)\times\{\tau\}}.
\end{align*}
Since $e>c(L)$, $\SSS_e$ satisfies the Palais-Smale condition. Therefore, any sequence $\gamma_n=(\Gamma_n,p_n)\in\W(\R/\Z,M)\times(\tau,\infty)$ such that $\SSS_e(\gamma_n)<\delta_0+1/n$ admits a subsequence converging to a local minimizer $\gamma=(\Gamma,p)\in\W(\R/\Z,M)\times(\tau,\infty)$ of $\SSS_e$, which is a waist.
\end{proof}

\begin{proof}[Proof of Theorem~\ref{t:contractible_waist}]
Since $M$ has finite fundamental group, its universal cover $\widetilde M$ has  finite degree over $M$, and in particular is compact. We denote by $\widetilde L:\Tan\widetilde M\to\R$ the lift of $L$. By \cite[Lemma~2.2]{Contreras:2002xw}, the Ma\~n\'e critical value does not change when a Lagrangian system is lifted to a finite cover, and therefore $c(L)=c(\widetilde L)$. Notice that a periodic curve $\gamma:\R/p\Z\to \widetilde M$ is a local minimizer of the free-period action functional associated to $\widetilde L$ if and only if its projection $\pi\circ\gamma$ is a local minimizer of $\SSS_{c(L)}$. Therefore, it is enough to prove the theorem for the Lagrangian $\widetilde L$, that is, from now on we can assume without loss of generality that $M$ is simply connected.

Since $c(L)>e_0(L)$, for all $\tau>0$ small enough, each solution $\gamma:[0,\tau]\to M$ of the Euler-Lagrange equation of $L$ with energy $E(\gamma(0),\dot\gamma(0))=c(L)$ satisfies  $\gamma(0)\neq\gamma(\tau)$, see \cite[Lemma~2.3(i)]{Asselle:2016qv}. In particular, for such $\tau>0$, the Aubry set $\Aubry(L)$ does not contain $\tau$-periodic orbits. Therefore, by Lemma~\ref{l:waist}, $L$ possesses at least a (contractible) waist with energy $e$ and period larger than $\tau$.
\end{proof}

\section{Periodic orbits of hyperbolic or twist type}
\label{s:hyp_twist}

\subsection{The cylinder of a non-degenerate periodic orbit}
\label{s:cylinder}

Let $(W,\omega)$ be a symplectic manifold of dimension $2d$, and $H:W\to\R$ a smooth Hamiltonian with associated Hamiltonian vector field $X_H$, given by $\omega(X_H,\cdot)=\diff H$, and Hamiltonian
 flow $\phi_H^t$. Consider a periodic orbit $\gamma:\R/\tilde p\Z\to H^{-1}(\tilde e)$, $\gamma(t)=\phi_H^t(\tilde z)$ on a regular energy hypersurface $H^{-1}(\tilde e)$. Any sufficiently small hypersurface $\Sigma\subset H^{-1}(\tilde e)$ that intersects $\gamma$ transversely at $\tilde z$ is a symplectic submanifold of $(W,\omega)$, and has a well defined associated first return-time function
\begin{align*}
\tau:\Sigma\to(\tilde p-\epsilon,\tilde p+\epsilon),
\qquad
\tau(z)=\min\{t>0\ |\ \Phi_H^t(z)\in\Sigma\}.
\end{align*}
The first return map
\begin{align*}
\Phi:(\Sigma,\omega)\to(\Sigma,\omega),
\qquad
\Phi(z)=\phi_H^{\tau(z)}(z),
\end{align*}
is a symplectomorphism. Its differential at the fixed point $\tilde z$ is the so-called Poincar\'e map $P:=\diff\Phi(\tilde z):\Tan_{\tilde z}\Sigma\to\Tan_{\tilde z}\Sigma$. 
The periodic orbit $\gamma$ is non-degenerate when 1 is not an eigenvalue of $P$. In such a case, $\gamma$ belongs to a so-called \textbf{orbit cylinder}: there exist smooth maps $e\mapsto z(e)\in W$ and $e\mapsto p(e)\in(0,\infty)$, for $e\in(\tilde e-\delta,\tilde e+\delta)$, such that $z(\tilde e)=\tilde z$, $p(\tilde e)=\tilde p$, $H(z(e))=e$, and each orbit $\gamma_e(t):=\phi_H^t(z(e))$ is $p(e)$-periodic (see, e.g., \cite[Prop.~2 on page 110]{Hofer:1994bq} for a proof of this fact). The periodic orbit $\gamma$ is \textbf{hyperbolic} if no eigenvalue of $P$ lies in the unit circle $S^1\subset\C$. Notice that, if $\gamma$ is hyperbolic, the same is true for the periodic orbits in a small orbit cylinder around $\gamma$.

\subsection{Periodic orbits of twist type}
Assume now that $\gamma$ is not hyperbolic. The tangent space at $\tilde z$ splits as the direct sum of $P$-invariant symplectic subspaces $\Tan_{\tilde z}\Sigma=\Tan_{\tilde z}^{\pm}\Sigma\oplus\Tan_{\tilde z}^{0}\Sigma$ in such a way that $P|_{\Tan_{\tilde z}^{\pm}\Sigma}$ has only eigenvalues outside the unit circle $S^1\subset\C$, whereas $P|_{\Tan_{\tilde z}^{0}\Sigma}$ has only eigenvalues on the unit circle. By our assumption, $\dim \Tan_{\tilde z}^{0}\Sigma>0$. The first return map $\Phi$ admits a so-called central manifold $\Wc\subset\Sigma$ at ${\tilde z}$, which is an embedded submanifold of $\Sigma$ that passes through $\tilde z$ with tangent space $\Tan_{\tilde z}\Wc=\Tan_{\tilde z}^{0}\Sigma$ and is invariant under $\Phi$ near $\tilde z$, see \cite[Theorem~3.3.5]{Klingenberg:1978so}. We denote by $S^1\subset\C$ the unit circle in the complex plane, and by $\sigma(P)$ the set of eigenvalues of the $P$. The Poincar\'e map $P$ is called \textbf{$4$-elementary} when for all quadruples of (not necessarily distinct) eigenvalues $e^{i\theta_1},e^{i\theta_2},e^{i\theta_3},e^{i\theta_4}\in\sigma(P)\cap S^1$ we have $e^{i(\theta_1+...+\theta_4)}\neq1$. In this case, the restriction $\Phi|_{\Wc}$ can by written in a so-called Birkhoff normal form near $\tilde z$: up to shrinking $\Wc$ around $\tilde z$, we can find symplectic coordinates $(z_1,...,z_q)=(x_1,y_1,...,x_q,y_q)$ on $\Wc$ such that, in these coordinates, $\tilde z=0$ and 
\begin{gather*}
\Phi|_{U}=(\Phi_1,...,\Phi_q):U\to\Wc,
\\
\Phi_j(z_1,...,z_q)=z_j\exp\left(a_j + 2\pi i\sum_{l=1}^q b_{jl}z_l\overline{z_{l}} \right) + w_j(z_1,...,z_q).
\end{gather*}
Here, $U\subset\Wc$ is a sufficiently small open neighborhood of $\tilde z$, the $a_j$'s and the $b_{jl}$'s are real numbers, and the maps $w_j:U\to\C$ vanish up to order $3$ at $\tilde z$. The fixed point $\tilde z$ is of \textbf{twist type} (or the periodic orbit $\gamma$ is of twist type) when the real $q\times q$ matrix $B:=(b_{jl})$ is non-singular. The celebrated Birkhoff-Lewis fixed point theorem \cite[Theorem~3.3.A.1]{Klingenberg:1978so} states that, if the fixed point $\tilde z$ of $\Phi$ is of twist type, then there exists a sequence $z_\alpha\to \tilde z$ as $\alpha\to\infty$ such that $z_\alpha$ is a periodic point of $\Phi$ with minimal period $p_\alpha$, and $p_\alpha\to\infty$ as $\alpha\to\infty$. In terms of the original Hamiltonian system, the statement can be rephrased as follows.

\begin{thm}[Hamiltonian Birkhoff-Lewis Theorem]\label{t:Birkhoff_Lewis}
Let $(W,\omega)$ be a symplectic manifold, $H:W\to\R$ a Hamiltonian, $e$ a regular value of $H$, and $\gamma:\R/p\Z\to H^{-1}(e)$ a periodic orbit of $H$ of twist type. Then, for any $p'>0$ there exists an arbitrarily small neighborhood $U\subset H^{-1}(e)$ of the support of $\gamma$ and a periodic orbit of $H$ contained in $U$ with minimal period larger than $p'$. In particular, the energy level $H^{-1}(e)$ contains infinitely many periodic orbits.
\hfill\qed
\end{thm}

\begin{rem}
\label{r:orbit_cylinder}
Since the Poincar\'e map of a periodic orbit $\gamma:\R/\tilde p\Z\to H^{-1}(\tilde e)$ of twist type is 4-elementary, all its eigenvalues on the unit circle have multiplicity 1 and are not real. In particular, 1 is not an eigenvalue of the Poincar\'e map, and therefore $\gamma$ belongs to an orbit cylinder $\gamma_e:\R/p_e\Z\to H^{-1}(e)$, $e\in(\tilde e-\delta,\tilde e+\delta)$. The condition of being $4$-elementary for a symplectic linear map is open, and the condition of being twist type is also open in the $C^\infty$ topology. Therefore, up to lowering $\delta>0$, all the periodic orbits $\gamma_e$ in the orbit cylinder are of twist type, and the Hamiltonian Birkhoff-Lewis Theorem apply to them.
\hfill\qed
\end{rem}

In view of the Birkhoff-Lewis fixed point theorem, in the quest for periodic orbits of a given Tonelli energy hypersurface one may assume that there is no periodic orbit of twist type. In such a case, generically, all the periodic orbits (if there are any) must be hyperbolic. This is guaranteed by the following result due to Carballo and Miranda \cite{Carballo:2013fc}, which extends a result of Klingenberg and Takens \cite{Klingenberg:1972dp} for geodesic flows. The notions of hyperbolicity or of twist type for the periodic orbits of a Tonelli Lagrangian are those inherited from the corresponding periodic orbits of the dual Tonelli Hamiltonian.

\begin{thm}[\cite{Carballo:2013fc}, Corollary~5]
\label{t:Carballo_Miranda}
Let $M$ be a closed manifold, $L:\Tan M\to\R$ a Tonelli Lagrangian, and $e\in\R$. Then there exists a residual subset $\UU\subset C^\infty(M)$ such that, for each $U\in\UU$, each periodic orbit of the Tonelli Lagrangian $L-U$ with energy $e$ is either hyperbolic or of twist type\footnote{In \cite{Carballo:2013fc}, the authors call ``weakly monotonous quasi elliptic'' a periodic orbit of twist type.}.
\hfill\qed
\end{thm}

\subsection{Hyperbolic periodic orbits of Tonelli Lagrangians}

Let $L:\Tan M\to\R$ be a Tonelli Lagrangian, with associated free period action functionals  $\SSS_e : \W(\R/\Z)\times(0,\infty)\to\R\cup\{+\infty\}$.
We denote by $\ind(\gamma)$ and $\nul(\gamma)+1$ the Morse index and the nullity of $\SSS_e$ at a critical point $\gamma=(\Gamma,p)$, and by $\ind_p(\gamma)$ and $\nul_p(\gamma)+1$ the Morse index and the nullity of the restricted functional $\SSS_e(\cdot,p)$ at the critical point $\Gamma$. Notice that $\ind_p(\gamma)\leq\ind(\gamma)\leq\ind_p(\gamma)+1$.
If $P$ denotes the Poincar\'e map associated to the periodic orbit $\gamma$, then 
\begin{align*}
\nul(\gamma)=\dim\ker(P-I),
\end{align*}
see \cite[Proposition~A.3]{Abbondandolo:2015lt}.
For each $m\in\N$, we denote by $\gamma^m=(\Gamma^m,mp)$, where $\Gamma^m(t)=\Gamma(mt)$, the $m$-th iterate of $\gamma$. Namely, $\gamma^m$ is $\gamma$ seen as an $mp$-periodic curve. If $\gamma$ is hyperbolic, Bott's iteration theory \cite{Bott:1956sp, Long:2002ed} implies that
\begin{align}
\label{e:Bott}
\ind_{mp}(\gamma^m)=m\,\ind_p(\gamma),
\qquad
\nul(\gamma^m)=0,
\qquad
\forall m\in\N.
\end{align}
In particular, the critical circle $\CC_\gamma:=\big\{(\Gamma(s+\cdot),p)\ \big|\ s\in\R/\Z\big\}$ is isolated in the set of critical points $\crit(\SSS_e)$. 
Notice that a hyperbolic periodic orbit $\gamma$ is a waist (that is,  $\CC_\gamma$ is a circle of local minimizers of $\SSS_e$) if and only if $\ind(\gamma)=0$, and in this case $\CC_\gamma$ has an arbitrarily small open neighborhood $\WW\subset \W(\R/\Z)\times(0,\infty)$ such that
\begin{align}
\label{e:nbhd_of_waist}
\inf_{\partial\WW} \SSS_e>\SSS_e(\gamma). 
\end{align}

\begin{lem}
\label{l:cylinder_hyperbolic_waist}
If $\gamma$ is a hyperbolic waist, the same is true for the periodic orbits in a small orbit cylinder around $\gamma$.
\end{lem}

\begin{proof}
Let $\gamma_{\tilde e}$ be a hyperbolic waist with energy $\tilde e$, which belongs to a hyperbolic cylinder $\gamma_e$, $e\in(\tilde e-\delta,\tilde e+\delta)$. We already remarked that, up to lowering $\delta>0$, all the periodic orbits $\gamma_e$  are hyperbolic. In particular, $\nul(\gamma_e)=0$ for all energy values  $e\in(\tilde e-\delta,\tilde e+\delta)$. This readily implies that $e\mapsto\ind(\gamma_e)$ is constant, and therefore $\ind(\gamma_e)=\ind(\gamma_{\tilde e})=0$.
\end{proof}

\section{Generic multiplicity results}
\label{s:mult}

We begin by proving Corollary~\ref{c:supercritical_mult} and Theorem~\ref{t:btz}, for which we have already introduced all the needed ingredients.

\begin{proof}[Proof of Corollary~\ref{c:supercritical_mult}]
By Carballo-Miranda's Theorem~\ref{t:Carballo_Miranda}, for each $e\in\R$ there exists a residual subset $\UU_e\subset C^\infty(M)$ such that, for each $U\in\UU_e$, each periodic orbit of the Tonelli Lagrangian $L-U$ with energy $e$ (if there is any) is either hyperbolic or of twist type. The countable intersection
\begin{align*}
\UU:=\bigcap_{e\in\Q} \UU_e
\end{align*}
is still a residual subset of $C^\infty(M)$. Let $U\in\UU$ be such that the Tonelli Lagrangian $L':=L-U$ satisfies $e_0(L')<c(L')$. From now on, all the arguments will be referred to the Lagrangian $L'$, which can be assumed to be fiberwise quadratic outside a large compact set without loss or generality (see Section~\ref{s:waists}), and whose free-period action functional at energy $e$ will be denoted by $\SSS_e':\MM\to\R$.

We choose an arbitrary energy value 
\[\tilde e\in(c(L'),\cw(L'))\cap\Q,\] 
where $\cw(L')$ is the constant given by Theorem~\ref{t:contractible_waist}. In order to complete the proof, we have to show that all the energy levels sufficiently close to $\tilde e$ possess infinitely many periodic orbits. If there exists a periodic orbit of twist type in the energy level $\tilde e$, the same is true on all energy levels close to $\tilde e$ (Remark~\ref{r:orbit_cylinder}), and we conclude by means of the Hamiltonian Birkhoff-Lewis Theorem (Theorem~\ref{t:Birkhoff_Lewis}). It remains to consider the case in which all the periodic orbits with energy $\tilde e$ are hyperbolic. By Theorem~\ref{t:contractible_waist}, there exists a contractible hyperbolic waist $\gamma_{\tilde e}$ with energy $\tilde e$. By Lemma~\ref{l:cylinder_hyperbolic_waist}, $\gamma_{\tilde e}$ belongs to an orbit cylinder $\gamma_e$, $e\in(\tilde e-\delta,\tilde e+\delta)$, such that each $\gamma_e$ is a hyperbolic waist. Let us fix an energy value $e\in(\tilde e-\delta,\tilde e+\delta)$, and simply call $\gamma:=\gamma_{e}$ in order to simplify the notation. Notice that $\SSS_{e}'(\gamma_{e})>0$, since $e>c(L')$, and therefore $\SSS_{e}'(\gamma^k)\to\infty$ as $k\to\infty$. Moreover, by~\eqref{e:Bott}, $\ind(\gamma^k)=\nul(\gamma^k)=0$ for all $k\in\N$, and therefore every iterate $\gamma^k$ is still a hyperbolic waist. We  consider the minmax values
\begin{align}
\label{e:minmax_sk_00}
s(k):=\inf_u \max_{s\in[0,1]} \SSS_e'(u(s)),
\end{align}
where the infimum ranges over the continuous maps $u:[0,1]\to\MM$ such that $u(0)=\gamma$ and $u(1)=\gamma^k$. By~\eqref{e:nbhd_of_waist}, we have the strict inequality $s(k)>\SSS_e'(\gamma^k)$, and in particular
\begin{align}
\label{e:s(k)_to_infty}
\lim_{k\to\infty} s(k)=+\infty.
\end{align}
Since $e>c(L')$, $\SSS_e'$ satisfies the Palais-Smale condition, and therefore $s(k)$ is a critical value of $\SSS_e'$ corresponding to at least a contractible periodic orbit. We denote by $\PP_k$ the intersection of the critical set $\crit(\SSS_e')\cap \SSS_e'^{-1}(s(k))$ with the connected component of $\MM$ of the contractible periodic curves.

Now, the end of the argument is a typical application of \cite[Theorem~2.6]{Abbondandolo:2014rb}, whose validity was extended to general Tonelli systems in \cite[Section~4.1]{Asselle:2016qv}. This theorem asserts that highly iterated periodic orbits are not mountain passes. More precisely, for each periodic orbit $\zeta$ with energy $e$, there exists $\overline m(\zeta)\in\N$ and, for each integer $m\geq\overline m(\zeta)$, an open neighborhood $\UU(\zeta^m)\subset\MM$ of the critical circle of $\zeta^m$ such that the inclusion is an injective map of connected components
\begin{align}
\label{e:injectivity_pi0}
\pi_0(\{\SSS_e'<\SSS_e'(\zeta^m)\}) \hookrightarrow \pi_0(\{\SSS_e'<\SSS_e'(\zeta^m)\}\cup\UU(\zeta^m)).
\end{align}

Assume by contradiction that there are finitely many geometrically distinct contractible periodic orbits $\zeta_{1},...,\zeta_{h}$ such that, for every $k\in\N$, every periodic orbit in $\PP_k$ is an iterate of one of them. We set 
\[\overline m:=\max\{\overline m(\zeta_1),...,\overline m(\zeta_h)\}.\]
By~\eqref{e:s(k)_to_infty}, if $k\in\K$ is large enough, every critical circle in $\PP_k$ is the critical circle of some iterated periodic orbit  $\zeta_i^m$ with $m\geq\overline m$. Let $\UU_k$ be the union of the open neighborhoods $\UU(\zeta_i^m)$, for all $\zeta_i^m\in\PP_k$. By means of gradient flow deformations, we can find a continuous map $u:[0,1]\to\MM$ such that $u(0)=\gamma^{nk}$, $u(1)=\gamma^{mk}$, and $u([0,1])\subset\{\SSS_e'<s(k)\}\cup\UU_k$. By~\eqref{e:injectivity_pi0}, we can modify $u$ into a continuous map $w:[0,1]\to\MM$ such that $w(0)=\gamma$, $w(1)=\gamma^{k}$, and $w([0,1])\subset\{\SSS_e'<s(k)\}$. The strict inequality $\max \SSS_e'\circ w<s(k)$ contradicts the definition of the minmax value~\eqref{e:minmax_sk_00}.
\end{proof}

\begin{proof}[Proof of Theorem~\ref{t:btz}]
As in the proof of Corollary~\ref{c:supercritical_mult}, by applying  Carballo-Miranda's Theorem~\ref{t:Carballo_Miranda} countably many times we obtain a residual subset of $\UU\subset C^\infty(M)$ such that, for all $U\in\UU$, each periodic orbit of the Tonelli Lagrangian $L':=L-U$ with any energy $e\in\Q$ is either hyperbolic or of twist type. We consider such an $L'$, which we will assume without loss of generality to be fiberwise quadratic outside a compact set (see Section~\ref{s:waists}), and its free-period action functionals $\SSS_e':\MM\to\R$. 

Let $\alpha\in\pi_1(M)$ be a homotopy class as in the statement, so that $\alpha$ is non-trivial and $\alpha^n=\beta\alpha^m\beta^{-1}$ for some distinct non-negative integers $n,m$ and for some $\beta\in\pi_1(M)$. Without loss of generality, we can assume that $n$ and $m$ are both strictly positive (indeed, if $m=0$, then $n>0$ and $\alpha^{n+1}=\alpha$).
For each $k\in\N:=\{1,2,3,...\}$, we denote by $\CC_k\subset\MM$ the connected component of those periodic curves freely homotopic to a representative of $\alpha^k$. Notice that the condition on $\alpha$ means precisely that 
\[\CC_{nk}=\CC_{mk},\qquad \forall k\in\N.\] 
Let $\K\subseteq\N$ be an arbitrary infinite subset. We set 
\[\CC:=\bigcup_{k\in\K} \CC_{nk}=\bigcup_{k\in\K} \CC_{mk}.\]

We choose an energy value $\tilde e\in(\cu(L'),\infty)\cap\Q$. If there exists a periodic orbit of twist type in $\CC$ with energy $\tilde e$, the Hamiltonian Birkhoff-Lewis Theorem (Theorem~\ref{t:Birkhoff_Lewis}) implies that there are infinitely many periodic orbits in $\CC$ on any energy level close to $\tilde e$. Therefore, we are left to consider the case in which all periodic orbits in $\CC$ with energy $\tilde e$ are hyperbolic.

Since $\tilde e>\cu(L')$, the free-period action functional $\SSS_{\tilde e}'$ is bounded from below on every connected component of $\MM$, and satisfies the Palais-Smale condition. Therefore, $\SSS_{\tilde e}'|_{\CC_1}$ admits a global minimizer $\gamma_{\tilde e}\in \CC_1$. Since $\gamma_{\tilde e}$ is a hyperbolic waist, by Lemma~\ref{l:cylinder_hyperbolic_waist} it belongs to a orbit cylinder of hyperbolic waists $\gamma_e\in\CC_1$, for  $e\in(\tilde e-\delta,\tilde e+\delta)\subset(\cu(L),\infty)$. We fix an arbitrary energy value $e\in(\tilde e-\delta,\tilde e+\delta)$, and we set $\gamma:=\gamma_{e}$. In order to complete the proof, we have to show that there are infinitely many periodic orbits with energy $e$ contained in $\CC$.

By~\eqref{e:Bott}, all iterates $\gamma^k$ are hyperbolic waists. Moreover, $\gamma^{nk}$ and $\gamma^{mk}$ are distinct points belonging to the same connected component $\CC_{nk}=\CC_{mk}$. For each $k\in\K$, we consider the minmax value
\begin{align}
\label{e:minmax_sk}
s(k):=\inf_u \max_{s\in[0,1]} \SSS_e'(u(s)),
\end{align}
where the infimum ranges over the continuous maps $u:[0,1]\to\MM$ such that $u(0)=\gamma^{nk}$ and $u(1)=\gamma^{mk}$. Since $e>\cu(L')$, $\SSS_e'$ satisfies the Palais-Smale condition, and therefore $s(k)>\max\{\SSS_e'(\gamma^{nk}),\SSS_e'(\gamma^{mk})\}$ is a critical value of $\SSS_e'$ corresponding to mountain pass critical points of $\SSS_e'$. We set \[\PP_k:=\crit(\SSS_e')\cap\SSS_e'^{-1}(s(k))\cap\CC_{nk},\] 
and claim that the family $\PP:=\cup_{k\in\K}\PP_k$ contains infinitely many geometrically distinct periodic orbits.

In order to prove this claim, we first show that, for each $\zeta\in\PP$, there are at most finitely many $k\in\K$ such that $\zeta\in\PP_k$. Indeed, assume by contradiction that 
\begin{align}\label{e:zeta_kj}
\zeta\in\bigcap_{i\in\N}\PP_{k_i}
\end{align}
for some infinite sequence $k_1<k_2<k_3<...$ in $\K$. If we set $s:=\SSS_e'(\zeta)$, we have 
\[s>\SSS_e'(\gamma^{nk_i})=k_i\SSS_e'(\gamma^n),\qquad\forall i\in\N,\] 
which implies that $\SSS_e'(\gamma^n)\leq 0$. If $\gamma$ has the form $\gamma=(\Gamma,p)\in\MM$, we thus have
\begin{align}\label{e:SSS_e'}
\SSS_{e'}'(\gamma^n)=\SSS_e'(\gamma^n)+(e'-e)p<\SSS_e'(\gamma^n)\leq 0,
\qquad
\forall e'\in(\cu(L),e).
\end{align}
However, \eqref{e:zeta_kj} implies that all the $\gamma^{nk_i}$ belong to the same connected component $\CC':=\CC_{nk_1}=\CC_{nk_2}=\CC_{nk_3}=...$, and therefore the inequality~\eqref{e:SSS_e'} implies 
\begin{align*}
 \inf_{\CC'}\SSS_{e'}'
 \leq
 \lim_{i\to\infty}\SSS_{e'}'(\gamma^{nk_i})
 =
 \lim_{i\to\infty}k_i \SSS_{e'}'(\gamma^{n})
 = -\infty,
\end{align*}
contradicting the fact that $\SSS_{e'}'$ is bounded from below on each connected component of its domain.

Now, assume by contradiction that there are only finitely many geometrically distinct periodic orbits $\zeta_1,...,\zeta_h$ with energy $e$ in $\CC$. The claim proved in the previous paragraph implies that, for all integers $\overline m>0$ there exists $\overline k>0$ such that, if $k\in\K$ is larger than $\overline k$, every critical circle in $\PP_k$ is the critical circle of some iterated periodic orbit $\zeta_i^m$ with $m\geq\overline m$. However, as in the proof of Corollary~\ref{c:supercritical_mult}, this is prevented by \cite[Theorem~2.6]{Abbondandolo:2014rb}.
\end{proof}

\begin{rem}
\label{r:reversibility}
If the Tonelli Lagrangian $L$ is reversible, meaning that $L(q,v)=L(q,-v)$ for all $(q,v)\in\Tan M$, then the assumption that $n,m$ are both non-negative in Theorem~\ref{t:btz} can be dropped. Indeed, if $\gamma\in\CC_1$ is the hyperbolic waist of the proof, the reversibility of $L$ guarantees that the curve $\overline\gamma(t):=\gamma(-t)$ is still a hyperbolic waist with the same energy as $\gamma$, and $\overline\gamma\neq\gamma$. If $n>0$ and $m<0$, the iterates $\gamma^n$ and $\overline\gamma^{-m}$ are distinct points in the same connected component $\CC_n$. Therefore, in order for the above proof to go through, it is enough to modify the min-max~\eqref{e:minmax_sk} and make the infimum range over the continuous maps $u:[0,1]\to\MM$ with $u(0)=\gamma^{nk}$ and $u(1)=\overline\gamma^{-mk}$.
\end{rem}

In order to prove Theorem~\ref{t:just_below}, we first need to recall the following result, which is a combination of Figalli-Rifford's closing lemma \cite[Theorem~1.2]{Figalli:2015ft} with an argument due to Contreras-Iturriaga \cite[page~935]{Contreras:1999wj} that makes a periodic orbit in the Aubry set hyperbolic.

\begin{thm}\label{t:Figalli_Rifford}
Let $L:\Tan M\to\R$ be a Tonelli Lagrangian with $e_0(L)<c(L)$. Then, there exists an arbitrarily $C^1$-small $U\in C^\infty(M)$ such that $c(L)=c(L-U)$ and the Aubry set of $L-U$ consists of exactly one hyperbolic periodic orbit.
\hfill\qed
\end{thm}

\begin{proof}[Proof of Theorem~\ref{t:just_below}]
By means of Theorem~\ref{t:Figalli_Rifford}, we choose an arbitrarily $C^1$-small $U\in C^\infty(M)$ such that $c(L)=c(L-U)$ and the Aubry set of $L':=L-U$ consists of exactly one hyperbolic periodic orbit $\gamma$. We denote by $\SSS_e':\MM\to\R\cup\{\infty\}$ and $\Phi_e':M\times M\to\R\cup\{-\infty\}$ respectively the free-period action functional and the action potential of $L'$ at energy $e$. Since $\gamma$ is the Aubry set $\Aubry(L')$, we have
\begin{align*}
\SSS_{c(L')}'(\gamma)=\Phi_{c(L')}'(\gamma(0),\gamma(0))=0,
\end{align*}
see Section~\ref{s:waists}. Therefore, since $\SSS_{c(L')}$ is a non-negative functional, $\gamma$ is a hyperbolic waist. By Lemma~\ref{l:cylinder_hyperbolic_waist}, $\gamma$ belongs to an orbit cylinder of hyperbolic waists $\gamma_e=(\Gamma_e,p_e)$, $e\in(c(L')-\epsilon,c(L')+\epsilon)$, with $\gamma=\gamma_{c(L')}$. Notice that
\begin{align*}
\tfrac{\diff}{\diff e} \SSS_e'(\gamma_e) = (\partial_e \SSS_e')(\gamma_e) = p_e >0,
\end{align*}
and therefore 
\begin{align*}
\SSS_e'(\gamma_e)<0,\qquad \forall e\in(c(L')-\epsilon,c(L')),\\
\SSS_e'(\gamma_e)>0,\qquad \forall e\in(c(L'),c(L')+\epsilon). 
\end{align*}
Since the iterates of each $\gamma_e$ have indices $\ind(\gamma_e^m)=\nul(\gamma_e^m)=0$, they are all hyperbolic waists, and their actions satisfy 
\begin{align*}
\lim_{m\to\infty} \SSS_e'(\gamma_e^m)
=
\left\{
  \begin{array}{@{}ccc}
    -\infty, &  & \mbox{if }e\in (c(L')-\epsilon,c(L')), \vspace{5pt}\\ 
    +\infty, &  & \mbox{if }e\in (c(L'),c(L')+\epsilon).\\ 
  \end{array}
\right.
\end{align*}
Therefore, we can complete the proof along the line of \cite[Section~3.3]{Abbondandolo:2014rb}. We sketch the argument for the reader convenience.

As we already did in the previous proofs, we can assume without loss of generality that our Lagrangian $L'$ is fiberwise quadratic outside a large compact subset of $\Tan M$ containing in particular the energy sublevel set $E^{-1}(-\infty,c(L')+\epsilon]$. For each $e<\cu(L')=c(L')$, the free-period action functional $\SSS_e'$ is unbounded from below in each connected component. Therefore, for each $e\in(c(L')-\epsilon,c(L'))$ we can define the minmax values
\begin{align*}
s_e(m):= \inf_u \max_{s\in[0,1]} \SSS_e'(u(s)),
\end{align*}
where the infimum ranges over the family of all continuous maps $u:[0,1]\to\MM$ such that $u(0)=\gamma_e^m$ and $\SSS_e'(u(0))<\SSS_e'(\gamma_e^m)$. One can show that $e\mapsto s_e(m)$ is monotone increasing, and therefore, even if $\SSS_e$ may not satisfy the Palais-Smale condition, a trick due to Struwe \cite{Struwe:1990sd} implies that $s_e(m)$ is a critical value of $\SSS_e$ for all energy values $e$ belonging to the full measure subset $I_m\subset(c(L')-\epsilon,c(L'))$ where the function $e\mapsto s_e(m)$ is differentiable. The countable intersection $I:=\cap_{m\in\N} I_m$ is still a full measure subset of $(c(L')-\epsilon,c(L'))$. 

Fix an arbitrary energy value $e\in I$. The well known argument of ``pulling one loop at the time'' due to Bangert implies that the difference $s_e(m)-\SSS_e'(\gamma_e^m)$ is uniformly bounded from above in $m$. Therefore, 
\begin{align}\label{e:going_down}
\lim_{m\to\infty} s_e(m)=-\infty.
\end{align}
If we assume by contradiction that there are only finitely many periodic orbits with energy $e$, Equation~\eqref{e:going_down} implies that for all $\overline n\in\N$ there exists $\overline m\in\N$ large enough so that, for all $m\geq \overline m$, every periodic orbit $\zeta$ with energy $e$ such that $\SSS_e(\zeta)=s_e(m)$ is the $n$-th iterate of some periodic orbit for some $n\geq\overline n$. Arguing as in the proof of Corollary~\ref{c:supercritical_mult}, we see that this is prevented by \cite[Theorem~2.6]{Abbondandolo:2014rb}, which asserts that highly iterated periodic orbits are not mountain passes.

This completes the proof of the theorem for the subcritical energy range $(c(L')-\epsilon,c(L'))$. The analogous argument proves the theorem for the supercritical energy range $(c(L'),c(L')+\epsilon)$; here, since the free-period action functional $\SSS_e'$ satisfies the Palais-Smale condition for all $e\in(c(L'),c(L')+\epsilon)$, there is no need to extract a full measure subset of $(c(L'),c(L')+\epsilon)$.
\end{proof}

\bibliography{_biblio}
\bibliographystyle{amsalpha}

\end{document}